\documentclass[11pt, a4paper, reqno]{amsart}

\usepackage[scaled]{helvet}
\usepackage{textcomp}
\usepackage{amssymb,amsmath,amscd,amsfonts,mathrsfs,yfonts}
\usepackage[T1]{fontenc}
\usepackage{amssymb}
\usepackage[left=3cm,right=3cm,top=2cm,bottom=2cm]{geometry}
\usepackage[matrix,arrow,curve,pic,2cell]{xy}
\usepackage{upgreek}
\usepackage{tkz-graph}
\usepackage{longtable}
\usepackage{pgfplots}
\usetikzlibrary{decorations.markings}
\usetikzlibrary{shapes.geometric}
\usetikzlibrary{patterns}
\usepackage{faktor}
\newtheorem{remark}{Remark}
\newtheorem{lemma}{Lemma}[section]
\newtheorem{theorem}{Theorem}[section]
\newtheorem{corollary}{Corollary}[section]
\newtheorem{definition}{Definition}

\DeclareMathOperator{\id}{id}

\newcommand{\RomanNumeralCaps}[1]
    {\MakeUppercase{\romannumeral #1}}

\begin{document}

\title[Some Poincar\'{e}--Sobolev inequalities for differential forms]{Some Poincar\'{e}--Sobolev inequalities  
\\ for differential forms}

\author{Vladimir Gol$'$dshtein}
\address{Department of Mathematics,
Ben Gurion University of the Negev,
P.O.Box 653, Beer Sheva, Israel} 
\email{vladimir@bgumail.bgu.ac.il}

\author{Yaroslav Kopylov}
\address{Sobolev Institute of Mathematics, 4 Koptyug Ave.,
630090, Novosibirsk, Russia} 
\email{yakop@math.nsc.ru; \quad yarkopylov@gmail.com}

\author{Roman Panenko}
\address{Department of Mathematics,
Ben Gurion University of the Negev,
Beer Sheva, Israel; \\
Sobolev Institute of Mathematics, 4 Koptyug Ave.,
630090, Novosibirsk, Russia} 
\email{panenkora@gmail.com}

\thanks{The work of Ya.~Kopylov and R.~Panenko was carried out in the framework of the State Task to the Sobolev
Institute of Mathematics (Project FWNF-2026-0026).}

\begin{abstract}
We continue the~study of embeddings between different classes of Sobolev
spaces of differential forms started in 2006 in a~paper by Gol$'$dshtein and
Troyanov. As in this paper, our study is based on relations
between $L_{q,p}$-cohomology and Sobolev type inequalities. The~main
results are estimates for the norms of the embedding operators for $q=p$
and $p>\frac{n-1}{k-1}$ in the~Euclidean $r$-ball $B(r)$ and its bi-Lipschitz
images. We also study the~compactness of such operators.

\vskip1pt

\end{abstract}

\maketitle

\section*{Introduction}

We begin with the~refined homotopy (averaging) operator $T$ constructed
by Iwaniec and Lutoborski (see \cite{IL93}) for a~convex bounded domain in $\mathbb{R}^{n}$. First we prove
the~compactness of this operator, which is in fact an embedding operator (see \cite{GT06}). In the~main part 
of the~present article, we deal with its weaker form, namely, the~usual homotopy operator~$S$ on a~ball in $\mathbb R^n$. This is in essence the~substance of the~Poincar\'{e} lemma, where the~operator induced by contraction to a~point is used. 

The~structure of the~paper is as follows.

In~Section~\ref{prel}, we prove the~compactness of the~Iwaniec--Lutoborski homotopy operator (see~\cite{IL93}) on the~space of $p$-integrable differential forms and discuss the~asymptotics of the~norms of the~homotopy
operator. 

In Section~\ref{homop}, 
we estimate the~norm of the~operator~$S$ in the Sobolev spaces of $p$-integrable differential $k$-forms for $p>\frac{n-1}{k-1}$ on the~open $r$-ball 
centered at the~origin (Theorem~\ref{ball}). 

In Section~\ref{gencon}, using a~modified version of the~homotopy operator,
we prove the corresponding result for bi-Lipschitz images of the unit
ball.

In Section~\ref{exam}, we estimate the~norm of~$S$ for the~standard simplex.

The Poincar\'{e}--Sobolev inequality for differential forms finds some applications in PDEs, especially 
to the~Laplace--Beltrami equation \cite{Rum}, the~Maxwell equation \cite{BS},
the~$p$-Laplace equation, and conformal spectral theory \cite{Stern}. Among geometric applications,
we should mention the~Stokes theorem for noncompact manifolds and semianalitic sets~\cite{Val}.

\section{Preliminaries}\label{prel}

Let $M$ be an~orientable Riemannian manifold.

\begin{definition}
Let $\Omega^k_p(M)$ be the~Banach space of (equivalence classes) of measurable differential 
$k$-forms with the~norm 
$$\|\omega\|_{\Omega^k_p(M)} = \Bigg\{ \int\limits_M |\omega|_x^p dx\Bigg\}^{\frac{1}{p}},
$$
where $|\omega|_x$ stands for the~modulus of a~differential form $\omega$
at a~point~$x$.
\end{definition}

Denote by $\Omega_0^k(M)$ the~space of smooth forms with compact support on~$M$ whose support is contained in~$\mathrm{int}\, M$.  

Below we understand the de Rham differential $d$ in the sense of distributions. Denote by $\Omega^k_{1,\mathrm{loc}}(M)$ the~space of differential $k$-forms whose modulus is integrable over any open bounded set in~$M$. Let $\omega \in \Omega^k_{1,\mathrm{loc}}(M)$. We refer to a~$(k+1)$-form $\eta$ as the~differential $d\omega$ of $\omega$ if
$$
\int_M \omega \wedge d\phi = (-1)^{k+1} \int_M  \eta \wedge \phi
$$
for every $\phi\in \Omega_0^{n-k-1}(M)$ with compact support contained in~$\mathrm{int}\, M$.  

\begin{definition}
We will use the~notation $\Omega^k_{p,\,p}(M)$ for the space of differential $k$-forms~$\omega$ such that 
$\omega\in \Omega^k_p(M)$ and $d\omega\in \Omega^{k+1}_p(M)$ with the~norm
$$\|\omega\|_{\Omega^k_{p,\,p}} = \|\omega\|_{\Omega^k_{p}}+\|d\omega\|_{\Omega^{k+1}_{p}}.$$
We also put
$$
Z_p^k(M) = \{ \omega\in \Omega^k_{p,\,p}(M) \,:\, d\omega=0 \};
$$
$$
B_p^k(M) = \{ \omega\in \Omega^k_p(M) \,:\, \omega=d\theta ~\text{for some}~
\theta\in \Omega^{k-1}_{p,\,p}(M)\}.
$$
\end{definition}

We denote by $\delta$ the~formal adjoint operator to~$d$.

In~\cite{IL93}, for~a~convex bounded domain~$D\subset\mathbb{R}^n$, Iwaniec and Lutoborski 
constructed 
a~homotopy operator $T:\Omega_{1,\mathrm{loc}}^k(D)\to \Omega_{1,\mathrm{loc}}^{k-1}(D)$, 
which can be expressed as follows. Let $\xi=(\xi_1,\dots,\xi_{k-1})\in \mathbb{R}^{k-1}$.
Then
$$
T\omega(x)\langle\xi\rangle = \int\limits_D 
\omega(z)\langle \zeta(z,x-z),\xi_1,\dots,\xi_{k-1}\rangle \, dz.
$$
Here the~function $\zeta:D\times \mathbb{R}^n \to \mathbb{R}^n$ is given by~the~formula
\begin{equation}\label{zeta}
\zeta(z,h)= \sum\limits_{\nu=k}^n \binom{n-k}{\nu-k} \frac{h}{|h|^\nu}
\int\limits_0^\infty s^{\nu-1} \varphi\left(z-s\frac{h}{|h|}\right) \, ds,
\end{equation}
where $\varphi\in C_0^\infty(D)$ is taken so that 
$$
\int_{\mathbb{R}^n} \varphi(y) \,dy=1 \,\, \text{and}  \,\, 
|\nabla\varphi|\le 2\mu(D)(\mathrm{diam} D)^{-n-1}.
$$
Then $\omega=d\,T\omega+Td\omega$ for~every $\omega\in \Omega_{1,\mathrm{loc}}^k(D)$.

\begin{theorem}\label{comp-op}
If $\frac{1}{p}-\frac{1}{q}<\frac{1}{n}$ then the operator
$$T :\Omega_p^k(D) \to \Omega_q^{k-1}(D)$$
is compact. 
\end{theorem}

\begin{proof}
As was observed in~\cite{IL93}, since $\varphi\left(z-s\frac{h}{|h|}\right) =0$
if $s > \mathrm{diam} D$, the~integration in~\eqref{zeta} is in~fact over the~finite 
interval $0\le s\le \mathrm{diam} D$.

If $\alpha=(i_1,\dots,i_k)$ is a~strictly increasing multi-index of degree $\deg\alpha=k$ then we put $dx^\alpha=dx^{i_1}\wedge\dots\wedge dx^{i_k}$.
We denote by $\mathfrak{M}^k$ the~set of all strictly increasing multi-indices
of degree~$k$.

Every form $\omega\in \Omega_p^k(D)$ is representable as 
$\omega = \sum_{\alpha\in\mathfrak{M}^k} \omega_\alpha \, dx^\alpha$. Let
$\{e_1,\dots,e_n\}$ be the~standard basis of~$\mathbb{R}^n$. For such $\omega$, we have
$$
T\omega(x) \langle \xi_1,\dots,\xi_{k-1} \rangle 
= \biggl\{\sum\limits_{\alpha\in \mathfrak{M}^k} \sum\limits_{i=1}^n  \int_D \biggl( K^i_k(z, x-z) \,
  \omega_\alpha(z)\, dz \biggr) \, dx^\alpha \biggr\} 
\langle e_i,\xi_1,\dots,\xi_{k-1} \rangle,
$$
where 
\begin{equation}\label{k_i}
K^i_k(x, h) := \sum\limits_{\nu=k}^n \binom{n-k}{\nu-k} \frac{ h_i }{|h|^\nu}
\int\limits_0^\infty s^{\nu-1} \varphi\left(z-s\frac{h}{|h|}\right) ds.
\end{equation}
Consider the~restriction $T_{(\alpha)}$ of~$T$ to the~subspace spanned by the $k$-forms 
$\omega= f dx^\alpha$ for a~fixed multi-index~$\alpha$, which is isometric to~$L^p(D)$.
For the~component $T_{(\alpha)}^\beta(f dx^\alpha)=T\omega(x) \langle e_{i_1},\dots,e_{i_{k-1}}\rangle $ 
(the coefficient at $dx^\beta$, $\beta=(i_1,\dots,i_{k-1})$), we have
$$
T\omega(x) \langle e_{i_1},\dots,e_{i_{k-1}} \rangle 
= \biggl\{ \sum\limits_{i=1}^n  \int_D \biggl( K^i_k(z, x-z) 
 \sum\limits_{\alpha\in \mathfrak{M}^k} f(z)\, dz \biggr) \, dx^\alpha \biggr\} 
\langle e_i,e_{i_1},\dots,e_{i_{k-1}} \rangle
$$
Each of the~operators $T_{(\alpha)}^\beta$ can be regarded as an~integral operator from $L^p(D)$
into $L^q(D)$ with kernel of the form~\eqref{k_i}.
Using relations~(23) in~\cite[p.~331]{KA82}, we conclude that each of the~operators $T_{(\alpha)}^\beta$ 
is compact if $\frac{1}{p}-\frac{1}{q}<\frac{1}{n}$ in~\cite[Theorem~6, p.~332]{KA82}. It follows that 
$T_{(\alpha)}$ is compact. 
\end{proof}

As V. Goldshtein and  M. Troyanov noted in \cite{GT06}, for the~open
$r$-ball $B(r)$ in $\mathbb{R}^n$ centered at the~origin and parameters $p$, $q$ satisfying 
$$
\frac{1}{p}-\frac{1}{q} < \frac{1}{n} 
$$
such an operator
$T :\Omega^{k+1}_p(B(r)) \to \Omega_q^{k}(B(r))$ is bounded with norm at most
$$ \||x|^{1-n} \|_{L^{\frac{pq}{p+pq-q}}(B(r))}.$$
So for $p= q$ that expression can be written  in more detail  as 
\begin{align*}
\||x|^{1-n} \|_{L^1} = \int\limits_{B(r)}|x|^{1-n} dx 
= \int\limits_0^r\int\limits_{S^{n-1}} \uprho^{1-n} \cdot \uprho^{n-1}d\uprho d\Upomega \\
=  \int\limits_0^r d\uprho \cdot \int\limits_{S^{n-1}} d\Upomega=    r\int\limits_{S^{n-1}} d\Upomega 
= r \cdot \frac{2\pi^{\frac{n}{2}}}{\Gamma\big(\frac{n}{2}\big)},
\end{align*}
where $d\Upomega$ is the~volume element of the~unit sphere~$S^{n-1}$. 
Using Stirling's formula for the gamma function (see  \cite[\RomanNumeralCaps{7}, \textsection{2.3}]{Bourb}),
$$
\Gamma(x) \sim \frac{x^x\sqrt{2\pi{x}}}{e^x},
$$
we can see that
$$
\frac{2\pi^{\frac{n+2}{2}}}{\Gamma\big(\frac{n+2}{2}\big)} \sim \frac{2\pi^{\frac{n+2}{2}} \cdot e^{\frac{n}{2}}}{\sqrt{\pi{n}} \cdot(\frac{n}{2})^\frac{n}{2}}.
$$
Thus, $\frac{2\pi^{\frac{n+2}{2}}}{\Gamma\big(\frac{n+2}{2}\big)}$ tends to zero as $n\to\infty$ faster
than $O\big((\frac{n}{2\pi^2})^{-\frac{n+1}{2}}\big)$.
This type of asymptotics could be useful for problems concerning
generalized coordinates in the configuration space of a system, where
$n$ could be pretty large.  

Below we consider the~properties of the~operator  
\begin{equation}\label{mainop}
S\omega(x)\langle \xi_1,\dots,\xi_{k-1}\rangle = \int\limits_0^1 t^{k-1}\omega(tx) 
\langle x,\,\xi_1,\dots,\,\xi_{k-1} \rangle dt
\end{equation}
acting on differential $k$-forms defined on the~open $r$-ball centered at the~origin.
Assuming for the simplification of computations that it acts on $\Omega_p^{k+1}(\mathbb{R}^{n+1}$), this operator has the~norm estimated by the~constant 
$$
C = \frac{  \sqrt{\binom{n+1}{k+1}}}{(pk - n)^{\frac{1}{p}}}.
$$
of Theorem~\ref{ball}. We infer
\begin{align*}
C^2 
 \sim &\frac{\sqrt{2\pi (n+1)} \cdot (n+1)^{n+1}}{ (pk - n)^{\frac{2}{p}} e^{n+1}}
 \frac{e^{k+1}}{\sqrt{2\pi (k+1)}\cdot (k+1)^{k+1}} \frac{e^{n-k+1}}{\sqrt{2\pi (n-k+1)}\cdot (n-k+1)^{n-k+1}}\\
 \sim& \frac{e}{(pk - n)^{\frac{2}{p}}\sqrt{2\pi}}\frac{(n+1)^\frac{2n+3}{2}}{(k+1)^\frac{2k+3}{2} (n-k+1)^\frac{2(n-k)+7}{2}}
\end{align*}
Finally, we have
$$
 \frac{  \sqrt{\binom{n+1}{k+1}}}{(pk - n)^{\frac{1}{p}}} 
 \sim \frac{\sqrt{e}}{(pk - n)^{\frac{1}{p}}(2\pi)^{\frac{1}{4}}}\frac{(n+1)^\frac{2n+3}{4}}{(k+1)^\frac{2k+3}{4} (n-k+1)^\frac{2(n-k)+7}{4}}
$$
Therefore, $C$ grows no faster than $O(n^{\frac{k}{2}-1})$ for suitable~$p$
(see the dependence of the constant on $p$ in the table ). Below we give the~values of this~constant 
for some parameters. 
\vskip 10pt
\begin{center}
\begin{longtable}{|c|c|c|c|}
\hline
\multicolumn{4}{|c|}{Dimension $n = 2$} \\
\hline
$k$&$p=2$ & $p=2.5$ & $p = 10 $\\
\hline
$1$&---  &--- &---\\
$2$& 1.0000& 0.8503& 0.8027\\
\hline
\multicolumn{4}{|c|}{Dimension $n = 3$} \\
\hline
$k$&$p=2$ & $p = 2.5$ & $p = 10$\\
\hline
$1$&--- & --- &---\\
$2$& ---&   2.2855&1.4069\\
$3$& 0.7071& 0.6444 &0.7490\\
\hline
\multicolumn{4}{|c|}{Dimension $n = 4$} \\
\hline
$k$&$p=2$ & $p=2.5$ &$p=10$\\
\hline
$1$& --- & ---  & ---\\
$2$&--- &  ---&2.0163\\
$3$&2.0000 &  1.5157&1.5066\\
$4$& 0.5774&0.5479  &0.7192\\
\hline
\multicolumn{4}{|c|}{Dimension $n = 5$} \\
\hline
$k$ &$p=2$ & $p=2.5$ & $p=10$ \\
\hline
$1$& --- & ---  & --- \\
$2$& --- & --- &2.6435\\
$3$& --- & 3.1623 & 2.3966\\
$4$& 1.5811 &  1.3548&1.6143\\
$5$& 0.5000&  0.4884&0.6988\\
\hline
\multicolumn{4}{|c|}{Dimension $n = 6$} \\
\hline
$k$ &$p=2$ & $p=2.5$ & $p=10$ \\
\hline
$1$& --- & ---  & --- \\
$2$& --- & ---  &3.2972\\
$3$& --- & ---&3.4112 \\
$4$& 3.8730&  2.6845&2.8071\\
$5$& 1.4142&1.2867 & 1.7166\\
$6$& 0.4472&  0.4467&0.6834\\
\hline
\multicolumn{4}{|c|}{Dimension $n = 7$} \\
\hline
$k$ &$p=2$ & $p=2.5$ & $p=10$ \\
\hline
$1$& --- & ---  & --- \\
$2$& --- & ---  &3.9894\\
$3$& --- & ---& 4.5438\\
$4$&--- & 5.0303 &4.3054\\
$5$&3.2404 & 2.6320&3.2208 \\
$6$&1.3229 &  1.2514&1.8122\\
$7$&0.4082&0.4152&0.6711\\
\hline
\multicolumn{4}{|c|}{Dimension $n = 10$} \\
\hline
$k$ &$p=2$ & $p=2.5$ & $p=10$ \\
\hline
$1$& --- & ---  & --- \\
$2$& --- & ---  &6.7082\\
$3$& --- & ---& 8.6189\\
$4$&--- &  --- &10.6878\\
$5$& ---  &15.8745&11.2607 \\
$6$& 14.4914 &8.7798 &9.9961\\
$7$&6.3246&5.3497&7.3932\\
$8$&3.0000&2.8500&4.4471\\
$9$&1.1952&1.2118 &2.0648\\
$10$&0.3333&0.3531&0.6444\\
\hline
\end{longtable}
\end{center}

Despite the fact that  $\frac{2\pi^{\frac{n+2}{2}}}{\Gamma\big(\frac{n+2}{2}\big)}$ tends to zero as $n\to\infty$
at a~high rate but $\frac{  \sqrt{\binom{n+1}{k+1}}}{(pk - n)^{\frac{1}{p}}}$ does not, for the low-dimensional 
case, we have  
\vskip 10pt
\begin{center}
\scalebox{0.85}{
\begin{tabular}{|c|c|c|c|c|c|c|c|c|c|}
\hline
n&2&3&4&5&6&7&8&9&10\\
\hline
$\frac{2\pi^{\frac{n+2}{2}}}{\Gamma\big(\frac{n+2}{2}\big)}$&6.2832&12.5664&19.7392&26.3189&31.0063&33.0734&32.4697&29.6866&25.5016\\
\hline
\end{tabular}
}
\end{center}
\vskip 10pt
We can see that the~constant  from \cite{GT06} is not a~nicely behaved one (compare the~values with the previous table) at least for low dimensions
\vskip 15pt
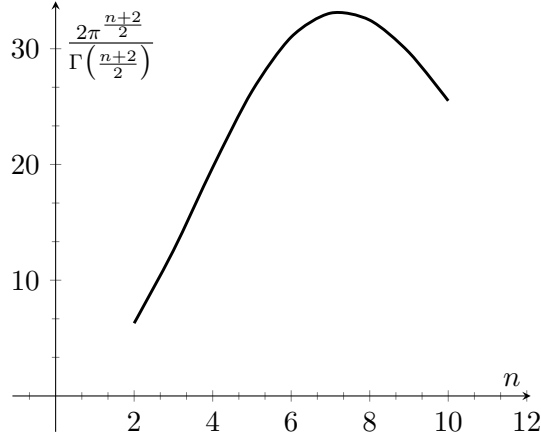
\begin{center}
\begin{tikzpicture}
\begin{axis}[
	axis lines=middle,
	enlargelimits=true,
	ymax = 31,
	ymin = 0,
	xmax = 11,
	xmin = 0,
	xlabel = {$n$},
	ylabel = {$\frac{2\pi^{\frac{n+2}{2}}}{\Gamma\big(\frac{n+2}{2}\big)}$},
	minor tick num = 2
]
\addplot[smooth,black, line width=1.1] table {
	x     y
       2   6.2832
      3   12.5664
      4	   19.7392
      5    26.3189
      6   31.0063
      7    33.0734
      8    32.4697
      9   29.6866
      10  25.5016
  };
\end{axis}
\end{tikzpicture}
\end{center}
\section{The Homotopy Operator}\label{homop}
In the present paper, we primarily work with spaces of differential forms defined on bounded domains 
in $\mathbb R^n$.  Moreover, we require such a~domain to be bi-Lipschitz homeomorphic to an open ball.    
As we said above, we begin with a homotopy operator on a ball in $\mathbb R^n$.
Given a differential $k$-form
$$
\omega = \sum_{j=(j_1,\dots,j_k)} f_{j_1,\dots,\,j_k} dx_{j_1}\wedge\dots \wedge dx_{j_{i}}\wedge \dots \wedge dx_{j_k}.
$$
Then, in the~local coordinates, the operator~\eqref{mainop} 
has the~form 
$$
S\omega = \sum_j \Bigg\{\int\limits_0^1f_{j_1,\dots,\,j_k}(tx)t^{k-1} dt \sum_i (-1)^{i}x_{j_i} dx_{j_1}\wedge\dots \wedge dx_{j_{i-1}}\wedge dx_{j_{i+1}}\wedge \dots \wedge dx_{j_k} \Bigg\}.
$$

In Theorem~\ref{interv} below, we deal with the~case of a~one-dimensional interval, whereas Theorem~\ref{ball} is concerned with the~case of a~ball
in~$\mathbb{R}^n$.

\begin{theorem}\label{interv}
If $]-r,\,r[$ is an~open interval and $\frac{1}{p}-\frac{1}{q} \le 1$ then 
$$
S \colon \Omega^1_p \big( ]-r,\,r[ \big)  \to \Omega^0_q\big( ]-r,\,r[ \big) 
$$
is well defined and
$$
\|S\omega \|_{\Omega^0_q \big(]-r,\,r[\big)} \le 2r\|\omega \|_{\Omega^1_p\big(]-r,\,r[\big)} .
$$
\end{theorem}

\begin{proof}
For the case $n=1$, the operator $S$ acts by the~formula
$$
S\colon\{f(x)dx\} \mapsto \Big\{ \int\limits_0^1 x f(tx)dt\Big\}.
$$
Consider the~constant function 
$$
{\bf 1} \colon x \mapsto 1, ~x\in B(r). 
$$
Then the expression for the norm of $S\omega$ looks as
\begin{align*}
\|S\omega\|_q^q &= \int\limits_{B(r)} \Bigg|  \int\limits_0^1 x f(tx)dt\Bigg|^q dx \le  \int\limits_{B(r)} \Bigg(  \int\limits_0^1 |x| \cdot |f(tx)|dt\Bigg)^q dx \\
& = \int\limits_{B(r)} \Bigg(  \int\limits_0^1  {\rm sgn}(x) \cdot {\bf 1}(x(1-t)) \cdot  |f(tx)| xdt\Bigg)^q dx.  
\end{align*}
Performing the~change of variable $z = tx$, we infer 
$$
 \int\limits_0^1  {\rm sgn}(x) \cdot {\bf 1}(x(1-t)) \cdot  |f(tx)| xdt =  \int\limits_0^x {\bf 1}(x-z) \cdot  |f(z)| dz \le  \int\limits_{B(r)} {\bf 1}(x-z) \cdot  |f(z)| dz,
$$
where the last inequality follows from the~relation ${\bf 1}(x-z) \cdot  |f(z)|  > 0$.
As a result, we obtain 
$$
\|S\omega\|_q^q  \le  \int\limits_{B(r)}  \Bigg(\int\limits_{B(r)} {\bf 1}(x-z) \cdot  |f(z)| dz \Bigg)^qdx = \||\omega|*{\bf 1}\|^q_q.
$$
Now we can apply Young's inequality. Suppose that $p,\,l,\,s \in [1,\infty]$ and 
$\frac{1}{p}+\frac{1}{l} = 1+ \frac{1}{s}$. Then 
$$
\|f*g\|_s\le\|f\|_p\|g\|_l.
$$
For the expression in question, we get
$$
\||\omega|*{\bf 1}\|_q \le \|{\bf 1}\|_l \||\omega|\|_p
$$
where 
$$
\frac{1}{p}+\frac{1}{l} = 1+\frac{1}{q}, 
$$
or, equivalently,  
$$
\frac{1}{p}-\frac{1}{q} =1- \frac{1}{l}.
$$
where $l \in [1,\,\infty]$. Given parameters $p$ and $q$ such that
$$
\frac{1}{p}-\frac{1}{q} \le 1,
$$
the operator $S$ satisfies the~inequality  
$$
\|S\omega\|_{\Omega^0_q \big(]-r,\,r[\big)} \le 2r\|\omega \|_{\Omega^1_p\big(]-r,\,r[\big)} 
$$

\end{proof}

\begin{theorem}\label{ball}
Let $B(r)\subset \mathbb R^n$, $n\ge 2$, be the~open $r$-ball centered at the~origin 
and $p> \frac{n-1}{k-1}$,~$k\ge2$. Then
$$
S \colon \Omega^{k}_p \big( B(r)\big)  \to \Omega^{k-1}_p\big( B(r) \big) 
$$
is well defined and
$$
 \|S\omega\|_{\Omega^{k-1}_p(B(r))} 
\le  \frac{  r\sqrt{\binom{n}{k}}}{(p(k-1) - n+1)^{\frac{1}{p}}}  \|\omega\|_{\Omega^{k}_p(B(r))}.
$$
\end{theorem}

\begin{proof}
Below in this proof, for convenience, we apply the operator $S$ to $(k+1)$-forms
$$
\omega = \sum_{j=(j_0,\dots,j_k)} f_{j_0,\dots,\,j_k} dx_{j_0}\wedge\dots \wedge dx_{j_{i}}\wedge \dots \wedge dx_{j_k},
$$
that is,  
$$
S \colon \Omega^{k+1}_p \big( B(r)\big)  \to \Omega^{k}_p\big( B(r) \big),
$$
$$
S\omega = \sum_j \Bigg\{\int\limits_0^1f_{j_0,\dots,\,j_k}(tx)t^{k-1} dt \sum_i (-1)^{i}x_{j_i} dx_{j_0}\wedge\dots \wedge dx_{j_{i-1}}\wedge dx_{j_{i+1}}\wedge \dots \wedge dx_{j_k} \Bigg\}.
$$
The~triangle inequality implies
$$
|S\omega|_x
\le \sum_j   \Bigg\{ \Big|  \int\limits_0^1 f_{j_0,\dots,\,j_k}(tx)  t^{k} dt \Big| \cdot \Big|\sum_i (-1)^i x_{j_i} dx_{j_0}\wedge\dots \wedge dx_{j_{i-1}}\wedge dx_{j_{i+1}}\wedge \dots \wedge dx_{j_k}\Big|_x\Bigg\}.
$$
Then we can calculate 
$$
\Big|\sum_i (-1)^i x_{j_i} dx_{j_0}\wedge\dots \wedge dx_{j_{i-1}}\wedge dx_{j_{i+1}}\wedge \dots \wedge dx_{j_k}\Big|_x =  \sqrt{\sum_i  x^2_{j_i}}
$$
Obviously,
$$
\sqrt{\sum_i  x^2_{j_i}} \le |x|.
$$
Hence,
\begin{align*}
|S\omega|_x 
& \le \sum_j   \Bigg\{ \Big|  \int\limits_0^1 f_{j_0,\dots,\,j_k}(tx)  t^{k} dt \Big| \cdot  \sqrt{\sum_i  x^2_{j_i}} \Bigg\} 
\\
& \le |x| \sum_j  \Big|  \int\limits_0^1 f_{j_0,\dots,\,j_k}(tx)  t^{k} dt \Big| 
\le |x| \int\limits_0^1    t^{k} \sum_j |f_{j_0,\dots,\,j_k}(tx)|  dt 
\end{align*}
Since $|x|< r$ for all points of the $r$-ball, we finally obtain 
$$
|S\omega|_x \le r \int\limits_0^1    t^{k} \sum_j |f_{j_0,\dots,\,j_k}(tx)|  dt
$$
Consider two finite-dimensional normed spaces $\big\langle \mathbb R^m,\,\|\cdot\|_{2}\big\rangle$ and $\big\langle \mathbb R^m,\,\|\cdot\|_{1}\big\rangle$.

For every $v \in  \mathbb R^m$, we have 
$$
\|v\|_{1}   \le  \sqrt{m}\|v\|_{2}
$$
Consider the~family of functions 
$$
f_{j_0,\dots,\,j_k} \colon B(r)\subset \mathbb R^n \to \mathbb R.
$$
It defines an $n$-dimensional surface in $\mathbb R^{\binom{n}{k+1}}$, which is realized as the~graph
$$
\mathcal F\colon x \mapsto  (f_{0,\dots,\,k}(x), \dots,\,f_{j_0,\dots,\,j_k}(x),\dots,\,f_{n-k-1,\dots,\,n}(x))
$$
and so, reproducing the previous argument, for every point 
$\mathcal F(x) \in \mathcal F(\mathbb R^n)\subset \mathbb R^{\binom{n}{k+1}}$, we obtain
$$
 \sum_j |f_{j_0,\dots,\,j_k}(x)| \le \sqrt{\binom{n}{k+1}} \sqrt{  \sum_j |f_{j_0,\dots,\,j_k}(x)|^2 } = \sqrt{\binom{n}{k+1}}|\omega|_x.
$$
And then
$$
|S\omega|_x \le r\int\limits_0^1    t^{k} \sum_j |f_{j_0,\dots,\,j_k}(tx)|  dt \le  r\sqrt{\binom{n}{k+1}}  \int\limits_0^1    t^{k}  |\omega|_{tx}  dt.
$$
Let us estimate the norm $\|S\omega\|_{\Omega_p(B(r))}$:
$$
\Bigg(\int\limits_{B(r)}|S\omega|_x^p dx\Bigg)^{\frac{1}{p}} \le r \sqrt{\binom{n}{k+1}} \Bigg( \int\limits_{ B(r)} \Big\{ \int\limits_0^1    t^{k}  |\omega|_{tx}  dt \Big\}^p dx\Bigg)^{\frac{1}{p}}
$$
Let us take a~look at
$$
\int\limits_{B(r)} \Big\{ \int\limits_0^1    t^{k}  |\omega|_{tx}  dt \Big\}^p dx.
$$
Applying Jensen's inequality in the~inner integral, we infer
\begin{align*}
\int\limits_{B(r)} \Big\{ \int\limits_0^1    t^{k}  |\omega|_{tx}  dt \Big\}^p dx \le \int\limits_{B(r)} \int\limits_0^1    t^{pk}  |\omega|^p_{tx}  dt  dx = \int\limits_0^1 \int\limits_{B(r)}   t^{pk}  |\omega|^p_{tx}dxdt. 
\end{align*}
We can consider the~map
$$
{ B(r)} \to {B(tr)}
$$
defined as
$$
x \mapsto t\cdot x,   \quad t\in [0,\,1].
$$ 
The~Jacobian matrix of this map is equal to
$$
J_t = \begin{bmatrix}
    t & &0 \\
    & \ddots & \\
   0 & & t
  \end{bmatrix}\, , 
$$
and 
$$
\det J_t = t^n.
$$
Let $k\ge1 $. Changing variables in the~inner integral, we have 

\begin{align*}
\int\limits_0^1 \int\limits_{B(r)}   t^{pk}  |\omega|^p_{tx}dxdt 
=  \int\limits_0^1 \int\limits_{B(tr)}   t^{pk - n}  |\omega|^p_{v}dvdt \le  \int\limits_0^1 t^{pk - n} dt \cdot \int\limits_{B(r)}     |\omega|^p_{v}dv
\end{align*}
Observe that
$$
 \int\limits_0^1 t^{pk - n} dt  = \frac{1}{pk - n+1}
$$
for
$$
pk-n>-1, ~p> \frac{n-1}{k}
$$
So we can conclude that 
$$
 \|S\omega\|_{\Omega^{k}_p(B(r))} \le \frac{  r\sqrt{\binom{n}{k+1}}}{(pk - n+1)^{\frac{1}{p}}} \|\omega\|_{\Omega^{k+1}_p(B(r))}
$$
for $p> \frac{n-1}{k}$.

This can be reformulated for $S$ acting on $k$-forms
$$
S \colon \Omega^{k}_p \big( B(r)\big)  \to \Omega^{k-1}_p\big( B(r) \big) 
$$
as follows:
$$
 \|S\omega\|_{\Omega^{k-1}_p(B(r))} 
\le  \frac{  r\sqrt{\binom{n}{k}}}{(p(k-1) - n+1)^{\frac{1}{p}}}  \|\omega\|_{\Omega^{k}_p(B(r))}.
$$

\end{proof}
\begin{remark} For the~case of 1-forms, we have 
$$
\|S\omega\|_{\Omega^0_p(B(r))} \le 
\frac{r^{1+\frac{n-1}{p}}\sqrt{n} }{(p+n-1)^{\frac{1}{p}}} \|\omega\|_{\Omega^1_p(B(r))}
$$
\end{remark}

\begin{proof}
By the~above considerations, we obtain 
$$
\Bigg(\int\limits_{B(r)}|S\omega|_x^p dx\Bigg)^{\frac{1}{p}} \le \sqrt{n} \Bigg( \int\limits_{ B(r)} \Big\{ \int\limits_0^1    |x| \cdot  |\omega|_{tx}  dt \Big\}^p dx\Bigg)^{\frac{1}{p}}.
$$
Therefore, it remains to estimate the~double integral. We infer
\begin{align*}
&\int\limits_{ B(r)} \Big\{ \int\limits_0^1    |x| \cdot  |\omega|_{tx}  dt \Big\}^p dx 
\le \int\limits_{ B(r)}  \int\limits_0^1    |x|^p \cdot  |\omega|_{tx}^p  dt\,  dx \\
&= \int\limits_{S^{n-1}} \int\limits_0^r 
\Big\{  \int\limits_0^1   \uprho^p \cdot  |\omega|_{tx(\uprho,\Upomega)}^p  dt\Big\} 
\uprho^{n-1}d\uprho\, d\Upomega  
= \int\limits_{S^{n-1}} \int\limits_0^r  
\int\limits_0^1  \uprho^{p+n-1} \cdot  |\omega|_{tx(\uprho,\Upomega)}^p  dt d\uprho\, d\Upomega \\
&=  \int\limits_{S^{n-1}} \int\limits_0^r  \uprho^{p+n-2} \Big\{ \int\limits_0^{\uprho} |\omega|_{\langle z,\,\Upomega\rangle}^p  dz \Big\} d\uprho\, d\Upomega 
=  \int\limits_0^r  \uprho^{p+n-2} \Big\{   \int\limits_{S^{n-1}} \int\limits_0^{\uprho} |\omega|_{\langle z,\,\Upomega\rangle}^p  dz\,d\Upomega  \Big\} d\uprho \\
&\le \|\omega\|^p_{\Omega^1_p(B(r))} \int\limits_0^r  \uprho^{p+n-2} d\uprho = \frac{r^{p+n-1} }{p+n-1} \|\omega\|^p_{\Omega^1_p(B(r))} 
\end{align*}
\end{proof}

\section{The~General Construction}\label{gencon}

Let $U$ and $V$ be subsets of $\mathbb R^n$ and let
$$
\omega = \sum_j f_{j_1,\dots,\,j_k} dy_{j_1}\wedge\dots \wedge dy_{j_{i}}\wedge \dots \wedge dy_{j_k} 
$$
be a differential $k$-form presented in local coordinates on $V$. 
Assume that 
$$
\varphi \colon U \to V
$$
is a~Lipschitz map 
$$
 \begin{pmatrix}
   x_1\\
   x_2\\
   \dots\\
   x_{n}
  \end{pmatrix}
\mapsto 
\begin{pmatrix}
   y_1\\
   y_2\\
   \dots\\
   y_{n}
  \end{pmatrix}
  $$
and $C$ is its Lipschitz constant, that is,
$$
|\varphi(x) -\varphi(x^{\prime})|\le C|x - x^{\prime}|.
$$  
It is well known that $\varphi$ has partial derivatives almost everywhere.
We put 
$$  
y_i = \varphi_i(x_1,\dots,\,x_n).
$$
Clearly,
$$
\frac{\partial \varphi_j }{\partial x_i}^2 \le \sum_j \frac{\partial \varphi_j }{\partial x_i}^2 = \Bigg| \frac{\partial}{\partial x_i}\varphi \Bigg|^2.
$$
Then the~relation 
$$
\frac{|\varphi(x) -\varphi(x^{\prime})|}{|x - x^{\prime}|}\le C
$$
implies  
$$
\Bigg| \frac{\partial \varphi_j }{\partial x_i}(x) \Bigg| \le C.
$$
We can represent the~pull-back of $\omega$ under $\varphi$ in local coordinates on $U$ as follows:
$$
\varphi^*\omega = \sum_j f_{j_1,\dots,\,j_k}(\varphi(x)) 
\Bigg\{ \sum_i \frac{\partial( \varphi_{j_1}, \dots, \varphi_{j_k} )}{\partial (x_{i_1},\dots,  x_{i_k})}(x) dx_{i_1}\wedge\dots \wedge dx_{i_k} \Bigg\}.
$$
Here $\frac{\partial( \varphi_{j_1}, \dots, \varphi_{j_k} )}{\partial (x_{i_1},\dots,  x_{i_k})}(x)$ stands for the~Jacobian of the~``numerator''
with respect to the~``denominator'' at a~point~$x$. 

\begin{lemma}
Let $U$ and $V$ be subsets of $\mathbb R^n$ and let 
$$
\varphi \colon U \to V
$$
be  a~ $C$-Lipschitz map. Then
$$
|\varphi^*\omega |_x\le \frac{n!\cdot C^k}{(n-k)!} \cdot|\omega|_{\varphi(x)} 
$$
for every differential $k$-form~$\omega$.
\end{lemma}

\begin{proof}
\begin{align*}
\varphi^*\omega 
&= \sum_j \sum_i  f_{j_1,\dots,\,j_k}(\varphi(x))\frac{\partial( \varphi_{j_1}, \dots, \varphi_{j_k} )}{\partial (x_{i_1},\dots,  x_{i_k})}(x) dx_{i_1}\wedge\dots \wedge dx_{i_k}  \\
& =\sum_i \Bigg\{\sum_j  f_{j_1,\dots,\,j_k}(\varphi(x))\frac{\partial( \varphi_{j_1}, \dots, \varphi_{j_k} )}{\partial (x_{i_1},\dots,  x_{i_k})}(x) \Bigg\} dx_{i_1}\wedge\dots \wedge dx_{i_k}  
\end{align*}

From this we obtain
\begin{align*}
|\varphi^*\omega |^2
&= \sum_i\Bigg\{ \sum_j f_{j_1,\dots,\,j_k} (\varphi(x))\frac{\partial( \varphi_{j_1}, \dots, \varphi_{j_k} )}{\partial (x_{i_1},\dots,  x_{i_k})}(x)\Bigg\} ^2\\
&\le \sum_i\Bigg\{ \sum_j  \Big| f_{j_1,\dots,\,j_k}(\varphi(x)) \Big| \cdot \Big|\frac{\partial( \varphi_{j_1}, \dots, \varphi_{j_k} )}{\partial (x_{i_1},\dots,  x_{i_k})}(x) \Big|\Bigg\} ^2\\
&\le \sum_i\Bigg\{ \sum_j  k! \cdot C^k \Big| f_{j_1,\dots,\,j_k}(\varphi(x)) \Big| \Bigg\} ^2
\end{align*}
Applying H\"{o}lder's inequality, we infer
\begin{align*}
|\varphi^*\omega |^2
&\le \sum_i\Bigg\{ \sum_j k! \cdot C^k \Big| f_{j_1,\dots,\,j_k}(\varphi(x)) \Big| \Bigg\} ^2\\
&\le \sum_i\Bigg\{ \Big( \sum_j k!^2 \cdot C^{2k} \Big)  \cdot \Big(\sum_j  \Big| f_{j_1,\dots,\,j_k}(\varphi(x)) \Big|^2 \Big)\Bigg\} \\
&\le  \binom{n}{k}\cdot k!^2 \cdot C^{2k}  \cdot \sum_i\sum_j  \Big| f_{j_1,\dots,\,j_k}(\varphi(x)) \Big|^2\\
&=   \binom{n}{k}^2 \cdot k!^2 \cdot C^{2k}  \cdot \sum_j  \Big| f_{j_1,\dots,\,j_k}(\varphi(x)) \Big|^2
\end{align*}
As a~result, we have
$$
|\varphi^*\omega |_x\le \frac{n!\cdot C^k}{(n-k)!} \cdot|\omega|_{\varphi(x)}.
$$
\end{proof}

Assume that there exists a~
homeomorphism $U \to V$ which can be presented as a~pair
of Lipschitz maps $\alpha$ and $\beta$ such that the~diagram
$$
\xymatrix{
&U\ar[dl]_{\beta}\ar[rr]^{\id}&&U\ar[dl]^{\beta}\\
 { V}\ar[rr]_{\id} &&  { V} \ar[ul]^{\alpha}
 }
$$
commutes, where
$$
|\beta(x) -\beta(x^{\prime})|\le C|x - x^{\prime}|
$$
and
$$
|\alpha(y) -\alpha(y^{\prime})|\le {C_1}|y - y^{\prime}|.
$$

As a result, we have the following commutative diagram for spaces of differential forms:
\begin{equation}\label{diag_forms}
\xymatrix@C = 0.5em{
&\Omega_p^*(U) \ar[rr]^{\id}\ar[dr]^{\alpha^*}&&\Omega_p^*(U)\\
  \Omega_p^*(V)\ar[ur]^{\beta^*}\ar[rr]_{\id} &&  \Omega_p^*(V) \ar[ru]_{\beta^*}
 }
\end{equation}

\begin{lemma}
The~operators in~\eqref{diag_forms} satisfy the~inequalities


$$
\|\beta^*\| \le   \frac{n!\cdot C^{k}
\cdot C_1^\frac{n}{p}}{(n-k)!}
$$
and
$$
\|\alpha^*\| \le  \frac{n!\cdot C_1^k\cdot C^{\frac{n}{p}}}{(n-k)!}.
$$
\end{lemma}

\begin{proof}
We have the~estimates 
$$
\Bigg| \frac{\partial \beta_j }{\partial x_i}(x) \Bigg| \le C
$$
and
$$
\Bigg| \frac{\partial \alpha_j }{\partial y_i}(y) \Bigg| \le C_1
$$
These estimates readily imply
$$
 |\det (d_x\beta)| \le  {C^n}
$$
and
$$
 |\det (d_y\alpha)| \le {C_1^n}.
$$
The~proof is completed by the~following calculations:


$$
\|\beta^*\omega\|_p^p = \int\limits_U |\beta^*\omega |^p_x dx 
\le \int\limits_U k!^p \binom{n}{k}^pC^{pk}\cdot|\omega|^p_{\beta(x)}dx;
$$

$$
dx = \det (d_y\alpha)\cdot dy; 
$$


$$
\|\beta^*\omega\|_p^p \le k!^p \int\limits_V  |\det (d_y\alpha)| \cdot  \binom{n}{k}^p \cdot C^{pk}\cdot |\omega|^p_y dy 
\le {k!^p}\cdot{C_1^n} \cdot \binom{n}{k}^p \cdot C^{pk} \int\limits_V |\omega|^p_y dy;
$$


$$
\|\beta^*\omega\|_p \le k!  \binom{n}{k}\cdot C^{k}\cdot C_1^\frac{n}{p}\cdot\|\omega\|_p 
=  \frac{n!\cdot C^{k}
\cdot C_1^\frac{n}{p}}{(n-k)!} \cdot \|\omega\|_p;
$$


$$
\|\alpha^*\eta\|_p^p = \int\limits_V |\alpha^*\eta|^p_y dy 
\le \int\limits_V  {  k!^p \binom{n}{k}^p } {C_1^{pk}}\cdot |\eta|^p_{\alpha(y)}dy;
$$

$$
 dy = \det (d_x\beta)\cdot dx; 
$$

$$
\|\alpha^*\eta\|_p^p \le k!^p \cdot\int\limits_U  |\det (d_x\beta)| \cdot {   \binom{n}{k}^p }\cdot{C_1^{pk}} \cdot |\eta|^p_x dx 
\le k!^p \cdot {C^n} \cdot  {   \binom{n}{k}^p } \cdot{C_1^{pk}}  \int\limits_U |\eta|^p_x dx;
$$

$$
\|\alpha^*\eta\|_p 
\le    k!^p \cdot \binom{n}{k}
\cdot C_1^k\cdot C^{\frac{n}{p}}\cdot\|\eta\|_p 
\le\frac{n!\cdot C_1^k\cdot C^{\frac{n}{p}}}{(n-k)!} \cdot \|\eta\|_p.
$$
\end{proof}
\begin{corollary}
Let $U\xrightarrow[]{\phi} V $ be a $C$-bi-Lipschitz homeomorphism. Then
$$
\|(\phi^{-1})^*\|,~\|\phi^*\| \le   \frac{n!\cdot C^\frac{pk+n}{p}}{(n-k)!}
$$
\end{corollary}

\begin{theorem}\label{est_for_dom}
Suppose that $U$ and $V$ are subsets in~$\mathbb R^n$  
that are $C$-bi-Lipschitz homeomorphic and $P:B_p^{k}(V)\to \Omega_p^{k-1}(V)$ is an~operator such that the~triangle
$$
\xymatrix@R = 1.5em{
& B_p^{k}( V) \ar[dd]^-{\id}\ar[ld]_{P}\\
\Omega_p^{k-1}( V) \ar[dr]_{d}&\\
&B_p^k(V)\\
}
$$
commutes and 
$
\|P\|\le M.
$
Then there exists an~operator $\gamma:B_p^{k}(U)\to \Omega_p^{k-1}(U)$ for which
the~triangle 
$$
\xymatrix@R = 1.5em{
& B_p^{k}( U) \ar[dd]^-{\id}\ar[ld]_{\gamma}\\
\Omega_p^{k-1}( U) \ar[dr]_{d}&\\
&B_p^k(U)\\
}
$$
commutes and
$$
\|\gamma\|  \le  
{M \cdot (n-k+1) }\cdot C^{\frac{2(pk+n)-p}{p}} 
\cdot \Bigg(\frac{n!}{(n-k+1)!}\Bigg)^2
$$
\end{theorem}

\begin{proof}
Combining the previous diagrams with the diagram for $P$
$$
\xymatrix@R = 3em{
\Omega_p^{k-1}(V)\ar[d]_-{\id}& B_p^{k}( V) \ar[d]^-{\id}\ar[l]_{P}\\
\Omega_p^{k-1}( V) \ar[r]_{d}& B_p^k(V)\,,\\
}
$$
we obtain
$$
\xymatrix@R = 3.5em{
\Omega^{k-1}_p(V)\ar[ddd]_{\id}\ar@{->}[rd]_{\beta^*}&&&B^{k}_p(V) \ar[ddd]^{\id}\ar[lll]_{P}\ar@{<-}[ld]^{\alpha^*}\\
&\Omega^{k-1}_p(U)\ar[d]_{\id} &B^{k}_p(U)\ar@{-->}[l]_{\gamma}\ar[d]^{\id}&\\
&\Omega^{k-1}_p(U)\ar[r]_{\tilde{d}}&B^{k}_p(U)&\\
\Omega^{k-1}_p(V)\ar[rrr]_{{d}}\ar@{->}[ru]^{\beta^*}&&&B^{k}_p(V)\ar@{->}[lu]_{\beta^*}
}
$$
The~operator 
$$
\gamma= \beta^* P \alpha^*
$$
has the~required properties. Indeed, $ {d} S = \id$ and $\tilde{d} \beta^*  = \beta^*  {d}$
due to the properties of the pull-back of differential forms. Using diagram chasing and an~elementary calculation, we have 
$$
\tilde{d}\id\gamma = \tilde{d} \beta^* P \alpha^* =  \beta^*  {d}  P \alpha^* = \beta^*  \alpha^* = \id
$$
and hence we have the~commutative diagram
$$
\xymatrix@R = 1.5em{
& B_p^{k}( U) \ar[dd]^-{\id}\ar[ld]_{\gamma}\\
\Omega_p^{k-1}(U) \ar[dr]_{d}&\\
&B_p^k(U)\,.\\
}       
$$
Moreover,
$$
\|\gamma\| = \| \beta^* P \alpha^*\| 
\le   {M \cdot (n-k+1) }\cdot C^{\frac{2(pk+n)-p}{p}} 
\cdot \Bigg(\frac{n!}{(n-k+1)!}\Bigg)^2
$$
since for each $\omega \in B^{k}_p(U)$ we have
\begin{align*}
\Big\| \beta^* \big(P (\alpha^*\omega)\big)\Big\|_{\Omega^{k-1}_p(U)}
\le& \|\beta^*\|\cdot \|P \big(\alpha^*\omega\big)\|_{\Omega^{k-1}_p(V)}\\
\le \|\beta^*\|\cdot \|P\|\cdot \|\alpha^*\omega\|_{\Omega^{k}_p(V)}
\le& \|\beta^*\|\cdot \|P\|\cdot \|\alpha^*\| \cdot \|\omega\|_{\Omega^{k}_p(U)}    
\end{align*}
\end{proof}

\section{An~Example}\label{exam} 
Now we can apply the result of the previous section to an open regular simplex $\Updelta$. 

Assume that
$$
\Updelta \xrightarrow{\varphi} B(R)
$$
maps the~simplex onto the circumscribed $R$-ball as follows
$$
x \mapsto \lambda(x)x,
$$
where 
$$
\lambda \colon \Updelta \to [1,\,n)
$$
is a~function periodic in the~general polar coordinates with respect to the~angular coordinates 
and monotone nondecreasing with respect to the~radial one. 
For example, we can use the function defined as follows.
Given $(n+1)$-simplex $ \Delta^{n+1}$, consider a face 
$$
[X_2,\dots,X_{n+1}]\hookrightarrow \Delta^{n+1},
$$
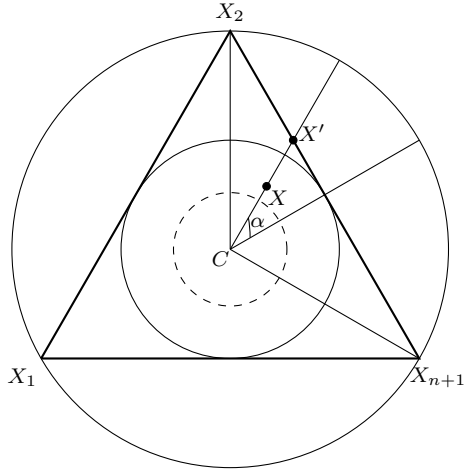
\begin{center}
\begin{tikzpicture}
\begin{scope}[scale = 5]
\draw[thick](-1/2,-1.73205/6)--(1/2,-1.73205/6)--(0,1.73205/3)--(-1/2,-1.73205/6);
\draw (0,0) circle(1.73205/3);
\draw (0,0) circle(1.73205/6);
\draw (0,0)--(0,1.73205/3);
\draw (0,0)--(1/2,-1.73205/6);
\draw (0,0)--(1/2,1.73205/6);
\draw (-1/2-0.05,-1.73205/6-0.05) node[font = \fontsize{8}{30}]{$X_1$};

\draw[dashed] (0,0) circle(0.15);
\draw (1.73205/6/3/2,1.73205/3*0.866/3/2) arc[start angle=10, end angle=0,radius=0.3]; 
\draw (1.73205/6/3/2+0.025,1.73205/3*0.866/3/2-0.013)  node[font = \fontsize{8}{30}]{$\alpha$};

\draw (0,1.73205/3+0.05) node[font = \fontsize{8}{30}]{$X_2$};
\draw (1/2+0.05,-1.73205/6-0.05) node[font = \fontsize{8}{30}]{$X_{n+1}$};

\draw (0,0)--(1.73205/6,1.73205/3*0.866);
\draw (-0.025,-0.025) node[font = \fontsize{8}{30}]{$C$};

\draw[fill, color = black] (1.73205/6/3,1.73205/3*0.866/3) circle(0.01);
\draw (1.73205/6/3+0.03,1.73205/3*0.866/3-0.03) node[font = \fontsize{8}{30}]{$X$};

\draw[fill, color = black] (1.73205/6/2/0.866,1.73205/3*0.866/2/0.866) circle(0.01);
\draw (1.73205/6/2/0.866+0.05,1.73205/3*0.866/2/0.866+0.02) node[font = \fontsize{8}{30}]{$X^\prime$};

\end{scope}
\end{tikzpicture}
\end{center}
Its barycenter $\overrightarrow{r}$  is specified by 
$$
\overrightarrow{r} = \frac{1}{n}\sum_{i=2}^{n+1} \overrightarrow{X}_i. 
$$
We intend to define a function $\lambda(X)$ such that,
for every direction $\overrightarrow{X}$, the function   
$$
\Lambda(t) = \lambda\left(t\frac{\overrightarrow{X}}{|\overrightarrow{X}|}\right), ~t \in[0,|X^\prime|]
$$
is piecewise linear and nondecreasing, taking the form of   
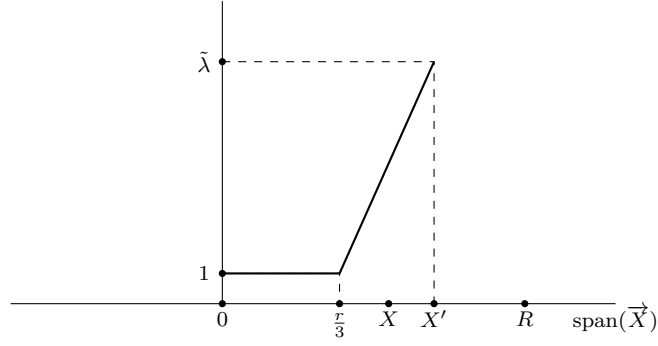
\begin{center}
\begin{tikzpicture}
\begin{scope}[scale = 4]

\draw (-1,0)--(1,0);
\draw (-0.3,0)--(-0.3,1);
\draw (1,-0.05) node[font = \fontsize{8}{30}]{$\rm{span}( \overrightarrow{X})$};
\draw[fill, color = black] (-0.3,0)circle(0.01);
\draw (-0.30,-0.05) node[font = \fontsize{8}{30}]{$0$};
\draw[fill, color = black] (0.4,0)circle(0.01);
\draw (0.4,-0.05) node[font = \fontsize{8}{30}]{$X^\prime$};

\draw[fill, color = black] (0.25,0)circle(0.01);
\draw (0.25,-0.05) node[font = \fontsize{8}{30}]{$X$};

\draw[fill, color = black] (0.7,0)circle(0.01);
\draw (0.7,-0.05) node[font = \fontsize{8}{30}]{$R$};

\draw[fill, color = black] (0.7/8,0)circle(0.01);
\draw (0.7/8,-0.07) node[font = \fontsize{8}{30}]{$\frac{r}{3}$};

\draw[fill, color = black] (-0.3,0.1)circle(0.01);
\draw  (-0.3-0.06,0.1) node[font = \fontsize{8}{30}]{$1$};

\draw[fill, color = black] (-0.3,0.8)circle(0.01);
\draw  (-0.3-0.06,0.8) node[font = \fontsize{8}{30}]{$\tilde{\lambda}$};

\draw[dashed] (-0.3,0.8)--(0.4,0.8)--(0.4,0);
\draw[dashed] (-0.3,0.1)--(0.7/8,0.1)--(0.7/8,0);
\draw[thick] (-0.3,0.1)--(0.7/8,0.1)--(0.4,0.8);
\end{scope}
\end{tikzpicture}
\end{center}
where $X^\prime \in \rm{span}( \overrightarrow{X}) \cap  [X_2,\dots,X_{n+1}]$ and
$$
\Lambda(X^\prime) =\tilde{\lambda}=  \frac{R}{|\overrightarrow{X^\prime}|}.
$$
In particular, for the direction $\overrightarrow{r}$ we can define
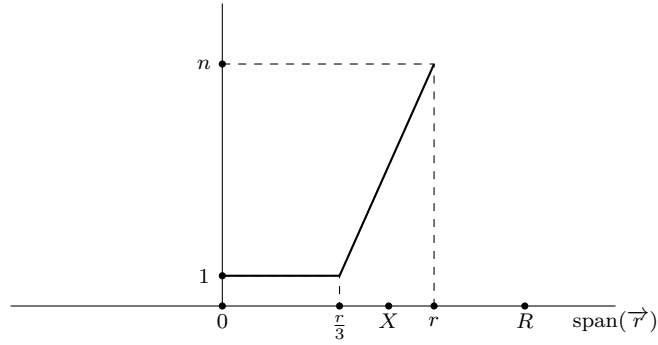
\begin{center}
\begin{tikzpicture}
\begin{scope}[scale = 4]

\draw (-1,0)--(1,0);
\draw (-0.3,0)--(-0.3,1);
\draw (1,-0.05) node[font = \fontsize{8}{30}]{$\rm{span}( \overrightarrow{r})$};
\draw[fill, color = black] (-0.3,0)circle(0.01);
\draw (-0.30,-0.05) node[font = \fontsize{8}{30}]{$0$};
\draw[fill, color = black] (0.4,0)circle(0.01);
\draw (0.4,-0.05) node[font = \fontsize{8}{30}]{$r$};

\draw[fill, color = black] (0.25,0)circle(0.01);
\draw (0.25,-0.05) node[font = \fontsize{8}{30}]{$X$};

\draw[fill, color = black] (0.7,0)circle(0.01);
\draw (0.7,-0.05) node[font = \fontsize{8}{30}]{$R$};

\draw[fill, color = black] (0.7/8,0)circle(0.01);
\draw (0.7/8,-0.07) node[font = \fontsize{8}{30}]{$\frac{r}{3}$};

\draw[fill, color = black] (-0.3,0.1)circle(0.01);
\draw  (-0.3-0.06,0.1) node[font = \fontsize{8}{30}]{$1$};

\draw[fill, color = black] (-0.3,0.8)circle(0.01);
\draw  (-0.3-0.06,0.8) node[font = \fontsize{8}{30}]{$n$};

\draw[dashed] (-0.3,0.8)--(0.4,0.8)--(0.4,0);
\draw[dashed] (-0.3,0.1)--(0.7/8,0.1)--(0.7/8,0);
\draw[thick] (-0.3,0.1)--(0.7/8,0.1)--(0.4,0.8);
\end{scope}
\end{tikzpicture}
\end{center}
since 
$$
\Lambda(r) = \frac{R}{r} = n.
$$
It is not hard to see that  for every direction $ \overrightarrow{X}$,  
$$
|\overrightarrow{X^\prime}| = \frac{r}{\cos \alpha},
$$
where $\alpha = \measuredangle( \overrightarrow{r}, \overrightarrow{X})$, we also have
$$ 
\cos\alpha = \Bigg\langle \frac{\overrightarrow{r}}{r}, \frac{\overrightarrow{X}}{|\overrightarrow{X}|} \Bigg\rangle.
$$
We can infer 
$$
\Lambda(t) = \frac{\tilde{\lambda}-1}{|\overrightarrow{X^\prime}|-r/3}\Big(t-\frac{r}{3}\Big) +1
$$
for $t \in [r/3,|\overrightarrow{X^\prime}|]$.
Each $X$ can be presented as
$$
\overrightarrow{X} = \overrightarrow{X}(t, \alpha),
$$
where $\alpha = \measuredangle( \overrightarrow{r}, \overrightarrow{X})$.
From now on, we can define 
$$
\overrightarrow{X} \mapsto \Lambda(t,\alpha) \cdot \overrightarrow{X}, 
$$
that is
$$
\overrightarrow{X} \mapsto  \left\{ \frac{3\cos\alpha(R\cos\alpha -r)}{r^2(3-\cos\alpha)}\Big(t-\frac{r}{3}\Big) +1\right\} \cdot \overrightarrow{X}.
$$
Or, to put it otherwise, we have
$$
\overrightarrow{X} \mapsto  
\left\{ \frac{ \langle \sum_{i=2}^{n+1} \overrightarrow{X}_i, \overrightarrow{X}\rangle \Big(R \langle\sum_{i=2}^{n+1} \overrightarrow{X}_i, \overrightarrow{X}\rangle  
-r^2 n|\overrightarrow{X}|\Big)}{r^2\Big(3rn|\overrightarrow{X}|-\langle \sum_{i=2}^{n+1} \overrightarrow{X}_i, \overrightarrow{X}\rangle\Big)}
\Big(3|\overrightarrow{X}|-r\Big) +1\right\} \cdot \overrightarrow{X}
$$
This function admits an extension by symmetries of the regular simplex. 
Fix $u,\,v \in \Updelta$. We have
\begin{align*}
u \mapsto\varphi(u) =  \lambda u,\\
v \mapsto \varphi(v) =  \lambda^\prime v. 
\end{align*}
Thus, $\varphi$ acts on the~plane spanned by vectors~$u$ and $v$ in such a way that if $\uptau$ is
the~triangle composed of the~vectors $u$, $v$, and $u-v$ then $\uptau_\varphi$ is
the~triangle  composed of the~vectors $\varphi(u)$, $\varphi(v)$, $\varphi(u)-\varphi(v)$:

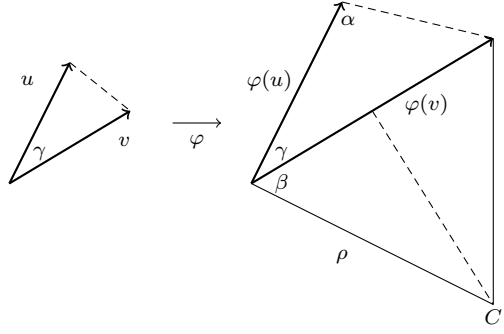
\begin{center}
\begin{tikzpicture}
\begin{scope}[scale = 0.8]
\draw[thick,->](0,0) -- (1,2);
\draw[thick,->](0,0) -- (2,1.2);
\draw[densely dashed,  color = black] (1,2)-- (2,1.2);
\draw  (0.3,1.7) node[font = \fontsize{8}{30}]  {$u$};
\draw  (1.9,0.7) node[font = \fontsize{8}{30}]  {$v$};
\draw  (0.5,0.5) node[font = \fontsize{8}{30}]  {$\gamma$};
\draw[->](2.7,1) -- (3.5,1);
\draw  (3.1,0.7) node[font = \fontsize{8}{30}]  {$\varphi$};
\draw[thick,->](4,0) -- (4+1.5,3);
\draw[thick,->](4,0) -- (4+4,2.4);
\draw[densely dashed,  color = black]  (4+1.5,3)-- (4+4,2.4);
\draw[densely dashed,](4+2,1.2) -- (4+4,-2);
\draw[](4,0) --(4+4,-2);
\draw[](4+4,2*1.2) -- (4+4,-2);
\draw (8,-2.2) node[font = \fontsize{8}{30}]  {$C$};
\draw  (4+1.6,3-0.3) node[font = \fontsize{8}{30}]  {$\alpha$};
\draw  (4.5,0) node[font = \fontsize{8}{30}]  {$\beta$};
\draw  (4.3,1.7) node[font = \fontsize{8}{30}]  {$\varphi(u)$};
\draw  (6.9,1.3) node[font = \fontsize{8}{30}]  {$\varphi(v)$};
\draw  (5.5,-1.2) node[font = \fontsize{8}{30}]  {$\rho$};
\draw  (4.5,0.5) node[font = \fontsize{8}{30}]  {$\gamma$};
\end{scope}
\end{tikzpicture}
\end{center}

Let $A_\uptau$ and $A_{\uptau_\varphi}$ be the areas of $\uptau$ and $\uptau_\varphi$ respectively.
Since area is a~tensor quantity, we have
$$
A_{\uptau_\varphi} =  \lambda \lambda^\prime A_\uptau.
$$
The elementary expression for the area of a triangle gives
$$
A_{\uptau} = \frac{\sin\zeta}{2}|u|\cdot|u-v|. 
$$
Therefore, 
$$
A_{\uptau} \le \frac{|u|\cdot|u-v|}{2}. 
$$
On the other hand, 
$$
A_{\uptau_\varphi}  = \frac{|\varphi(u)|\cdot|\varphi(v)|\cdot|\varphi(u)-\varphi(v)|}{4\rho},
$$ 
where $\rho$  is the radius of  the circumscribed circle of the triangle $\uptau_\varphi$.
Consequently,
$$
|\varphi(u)-\varphi(v)| =  \frac{4\rho A_{\uptau_\varphi} }{|\varphi(u)|\cdot|\varphi(v)|}
$$
and hence
$$
|\varphi(u)-\varphi(v)| =  \frac{4\rho A_{\uptau_\varphi} }{ \lambda |u|\cdot |\varphi(v)|} =   \frac{4\rho\lambda \lambda^\prime A_\uptau }{ \lambda |u|\cdot |\varphi(v)|}   
\le\frac{2\rho\lambda \lambda^\prime |u|\cdot|u-v| }{\lambda|u|\cdot |\varphi(v)|}   =  \lambda^\prime\frac{2\rho}{|\varphi(v)|} |u-v|  
$$
Let $\alpha$ be the~angle between $\varphi(u)$, $\varphi(u)-\varphi(v)$. Then
$$
\beta = \frac{\pi}{2} - \alpha.
$$
We have
$$
\cos\beta = \frac{\varphi(v)}{2\rho}
$$
and so
$$
\frac{2\rho}{|\varphi(v)|} = \frac{ 1}{\cos\beta}.
$$
If $u = v+h$, where $h \to 0$, then $\alpha \to \frac{\pi}{2}$ and $\beta \to 0$,
$$
\lim_{h\to 0} \frac{|\varphi(u)-\varphi(v)|}{|u-v|}  =  \lim_{h\to 0} \frac{\lambda^\prime}{\cos\beta} \le \lim_{h\to 0} \frac{n}{\cos\beta} = n
$$
since any $\lambda \le n$. So for the~ almost everywhere differentiable map
$$
\varphi \colon \Updelta \to B(R),
$$
acting by the~rule
$$
\begin{pmatrix}
   x_1\\
   x_2\\
   \dots\\
   x_{n}
  \end{pmatrix}
\mapsto 
 \begin{pmatrix}
   y_1\\
   y_2\\
   \dots\\
   y_{n}
  \end{pmatrix}\,, 
   $$
  where 
$$  
y_j = \varphi_j(x_1,\dots,\,x_n),
$$
we have
$$
\frac{\partial \varphi_j }{\partial x_i}^2 \le \sum_j \frac{\partial \varphi_j }{\partial x_i}^2 = \Bigg| \frac{\partial}{\partial x_i}\varphi \Bigg|^2.
$$
Therefore,
$$
\Bigg| \frac{\partial \varphi_j }{\partial x_i}(x) \Bigg| \le n.
$$
Observe that $\varphi$ is a~Lipschitz map. Indeed, let us look at
$$
|\varphi(u) - \varphi(v)|^2 = \sum\limits_i |\varphi_i(u) - \varphi_i(v)|^2.
$$
It is not hard to see that 
\begin{align*}
\varphi_i(u) - \varphi_i(v) = \varphi_i(u+t(v-u))\big|_{t=0} - \varphi_i(u+t(v-u)))\big|_{t=1} \\
= \int\limits_0^1\frac{d}{dt} \varphi_i(u+t(v-u)) dt = \int\limits_0^1 \sum\limits_j (v_j -u_j) \frac{\partial \varphi_i}{\partial x_j} dt
\end{align*}
Then, using Jensen's inequality, we obtain
$$
(\varphi_i(u) - \varphi_i(v))^2 \le n^2  n^2 \frac{1}{n}\sum\limits_j (v_j -u_j)^2 
$$
Hence
$$
|\varphi(u) - \varphi(v)|^2 \le n^3 \sum\limits_i \sum\limits_j (v_j -u_j)^2 = n^4 |u - v|^2
$$
As a~result, $\varphi$ is a~Lipschitz map with constant at~most~$n^2$.
Also we can see that $\varphi^{-1}$ is a~Lipschitz map with constant $1$.
Finally, we have
$$
\|\mathscr{S}\omega\|_{ \Omega^{k-1}_p(\Updelta)} 
\le  \frac{ n^{\frac{4(pk+n)-2p}{p}} \cdot \sqrt{\binom{n}{k}} \cdot (n-k+1)}{ (p(k-1) - n+1)^{\frac{1}{p}}}  \cdot \Bigg(\frac{n!}{(n-k+1)!}\Bigg)^2 \|\omega\|_{ \Omega^{k}_p(\Updelta)},
$$ 
where
$$
\mathscr{S}= (\varphi^{-1})^* S \varphi^*.
$$

\smallskip

We have proved the~following assertion:

\begin{theorem}\label{main_op}
Let $\Updelta$ be an open regular $n$-simplex,
$p> \frac{n-1}{k-1}$, $k\ge2$, and 
$\mathscr{S}:\Omega^{k}_p(\Updelta) \to \Omega^{k-1}_p(\Updelta)$
be the~operator defined above. Then, for every $\omega \in B^{k}_p(\Updelta)$, we have
$$
\|\mathscr{S}\omega\|_{ \Omega^{k-1}_p(\Updelta)} 
\le  \frac{ n^{\frac{4(pk+n)-2p}{p}} \cdot \sqrt{\binom{n}{k}} \cdot (n-k+1)}{ (p(k-1) - n+1)^{\frac{1}{p}}}  \cdot \Bigg(\frac{n!}{(n-k+1)!}\Bigg)^2 \|\omega\|_{ \Omega^{k}_p(\Updelta)}.
$$ 
\end{theorem}



\end{document}